\documentclass[12pt]{scrartcl}
\usepackage{amsmath,amsfonts,amssymb,amsthm}
\usepackage{mathtools}
\usepackage{xfrac}
 \usepackage{tabularx}
 \usepackage{hyperref}
\usepackage{color}
\newcommand{\avint}{\int \hspace{-1.0em}-\,}
\usepackage{mathabx}
\mathtoolsset{showonlyrefs}

\newcommand{\mel}{\MoveEqLeft}

\newcommand{\grad}{\nabla}

\newcommand{\laplace}{\Delta}
\newcommand{\Nu}{\mathit{Nu}}

\newcommand{\Ra}{\mathit{Ra}}

\renewcommand{\div}{\grad\cdot}

\newcommand{\N}{\mathbb{N}}

\newcommand{\R}{\mathbb{R}}

\newcommand{\Z}{\mathbb{Z}}

\usepackage{sectsty}

\sectionfont{\large}
\subsectionfont{\normalsize}
\subsubsectionfont{\normalsize}
\paragraphfont{\normalsize}

\newcommand{\la}{\langle}
\newcommand{\ra}{\rangle}

\newcommand{\dd}{{\mathrm d}}

\def\Xint#1{\mathchoice
{\XXint\displaystyle\textstyle{#1}}%
{\XXint\textstyle\scriptstyle{#1}}%
{\XXint\scriptstyle\scriptscriptstyle{#1}}%
{\XXint\scriptscriptstyle\scriptscriptstyle{#1}}%
\!\int}
\def\XXint#1#2#3{{\setbox0=\hbox{$#1{#2#3}{\int}$ }
\vcenter{\hbox{$#2#3$ }}\kern-.59\wd0}}
\def\avint{\Xint-}

\newtheorem{prop}{Proposition}
\newtheorem{theorem}{Theorem}
\newtheorem{lemma}{Lemma}

\begin{document}

\title{\LARGE Infinite Prandtl number convection with Navier-slip boundary conditions}
\author{\large Christian Seis\thanks{Institut f\"ur Analysis und Numerik,   Universit\"at M\"unster, Orl\'eans-Ring 10, 48149 M\"unster, Germany. Email: seis@uni-muenster.de}}
\setkomafont{author}{\sffamily}
\setkomafont{date}{\sffamily}
%{Institut f\"ur Analysis und Numerik,   Universit\"at M\"unster,
%Orl\'eans-Ring 10, 48149 M\"unster, Germany.}

\maketitle

\begin{abstract}
We are concerned with infinite Prandtl number Rayleigh--B\'enard convection with Navier-slip boundary conditions. The goal of this work is to estimate the average upward heat flux measured by   the non-dimensional Nusselt number $\Nu$ in terms of the Rayleigh number $\Ra$, which is a non-dimensional quantity measuring the imposed temperature gradient.  We derive bounds on the Nusselt number that coincide for relatively small slip lengths with the optimal Nusselt number scaling  for   no-slip boundaries, $\Nu\lesssim \Ra^{1/3}$; for relatively large slip lengths, we recover scaling estimates for   free-slip boundaries, $\Nu\lesssim \Ra^{5/12}$.
 \end{abstract}
 
 \textbf{Statements and Declarations:} The author states that there is no conflict of interest. Data sharing is not applicable to this article as no data sets were generated or analysed during the current study.

\section{Introduction}
Rayleigh--Bénard convection refers to the heat transfer and fluid motion that occurs when a layer of fluid is heated from below and cooled from above. This phenomenon is observed in various fields of physics, including oceanography, atmospheric science, and astrophysics. Characterised by boundary layers, steady convection rolls, and chaotic and turbulent behaviour, as well as stable and unstable parameter regimes, Rayleigh--Bénard convection has become a paradigm in theoretical and experimental fluid dynamics. We refer to the reviews by Siggia \cite{Siggia1994} and Manneville \cite{Manneville2006} for further reading.

A key feature in Rayleigh--B\'enard convection is the enhancement of the heat transport across the layer due to  thermal convection. The non-dimensional quantity measuring the ratio of the total heat flux to the purely conductive heat flux is the Nusselt number $\Nu$. Its dependence on the applied temperature gradient, measured in terms of the non-dimensional Rayleigh number $\Ra$ is of particular interest, see, for instance, \cite{AhlersGrossmannLohse09}.

In the regime of infinite Prandtl numbers, which refers to situations in which inertia is negligible compared to viscous forces, and imposing no-slip boundary conditions on the horizontal boundaries, a marginally stable boundary layer argument of Malkus  suggest the scaling relation $\Nu\sim \Ra^{1/3}$ in the turbulent regime $\Ra\gg1$ between the Nusselt number and the Rayleigh number \cite{Malkus54}. First rigorous bounds on the Nusselt number were obtained from the 1960ies on \cite{Howard63,Busse70,Howard72,Busse79,DoeringConstantin96}. Most notably, Constantin and Doering  \cite{ConstantinDoering99} in 1999  derived the slightly suboptimal bound $\Nu\lesssim \Ra^{1/3}\log^{2/3}\Ra$ by exploiting the maximum principle for the temperature and estimating singular integral kernels that relate the vertical component of the velocity to the temperature distribution. This bound was later improved by Doering, Otto and Reznikoff \cite{DoeringOttoReznikoff06} in 2006  to $\Nu\lesssim \Ra^{1/3}\log^{1/3}\Ra$ by using the background field method \cite{DoeringConstantin94,DoeringConstantin96}. The background field method is a mathematically quite simple tool that is based on a stability estimate for a given background temperature  profile. Any stable profile yields an upper bound  on the Nusselt number, and physically relevant bounds were considered by some authors as a rigorous justification on Malkus' boundary layer theory. However, it was proved by Nobili and Otto \cite{NobiliOtto17} that using the background field method, it is not possible to go beyond the bound $\Nu\lesssim Ra^{1/3} \log^{1/15}\Ra$. Instead Otto and the author \cite{OttoSeis11} in 2011 could further elaborate on Constantin and Doering's strategy to derive a double logarithmic improvement, $\Nu \lesssim \Ra^{1/3}\log^{1/3}\log\Ra$, implying, in particular, that the background field method does not carry physical relevance beyond being a mathematical technique for deriving scaling estimates. Only recently, the logarithmic correction could be completely removed in the work of Chanillo and Malchiodi \cite{ChanilloMalchiodi24}, who explored more sophisticated harmonic analysis tools to further improve on \cite{OttoSeis11}.

In certain flow situations, no-slip boundary conditions are considered to be inaccurate, such as for high altitude atmospheric gases and  microscale fluids. If, instead of no-slip boundary conditions, free-slip (no-stress) boundary conditions are relevant at the horizontal walls of the layer, the fluid flow is less restricted and and the heat transport enhances up to $\Nu\sim \Ra^{5/12}$. This scaling was   observed in numerical simulations by Otero \cite{Otero2002} and rigorous upper bounds were established in \cite{WhiteheadDoering11,WhiteheadDoering12}.

In the present work, our aim is to study  infinite Prandtl number Rayleigh--B\'enard convection with Navier-slip boundary conditions, which, in  a certain sense, interpolate between the no-slip and the free-slip boundary conditions. Indeed, the Navier-slip boundary conditions linearly relate the shear velocity to the tangential stress, their ratio being the (dimensionless)  slip length $\sigma$, so that $\sigma=0$ leads to no-slip boundary conditions while $\sigma=\infty$ gives free-slip boundary conditions. This model thus introduces one additional parameter, thanks to which it becomes applicable to a wide range of fluid models covering any intermediate boundary condition between no-slip and free-slip.

 In our main result, we obtain that interpolation on the level of the Nusselt number.

\begin{theorem}
\label{T1}
Let $\sigma>0$ be given. In the regime $\Ra\gg1$, the Nusselt number satisfies the bounds
\[
\Nu \lesssim \begin{cases} (\Ra)^{1/3}&\mbox{for }\sigma\lesssim \Ra^{-1/3},\\
(\sigma \Ra)^{1/2}&\mbox{for } {\Ra^{-1/3}}\lesssim \sigma \lesssim  {\Ra^{-1/6}},\\
\Ra^{5/12}& \mbox{for }\sigma \gtrsim  {\Ra^{-1/6}}.
\end{cases}
\]
\end{theorem}

Apparently, in  the regime $\sigma \lesssim \Ra^{-1/3}$, we recover the  Nusselt number scaling that is optimal in the no-slip model, while for  $\sigma \gtrsim \Ra^{-1/6}$, we recover the Nusselt number scaling for free-slip boundaries. The intermediate $(\sigma \Ra)^{1/2}$ scaling is peculiar to the Navier-slip boundary conditions. To the best of the author's knowledge, it has not been reported in the literature and it would be interesting  to see if it can be reproduced in experiments or numerical simulations.  We caution the reader that this intermediate scaling  \emph{must not} be confused with the ultimate $Nu\sim Ra^{1/2}$ scaling, as it applies only to relatively small slip lengths, and the exponent on the Rayleigh number in our Nusselt bounds will never exceed the free-slip exponent $5/12$.

Navier-slip Rayleigh--B\'enard convection attracted some interest in recent years. Nusselt bounds in the finite Prandtl number case were derived in \cite{DrivasNguyenNobili22,BleitnerNobili24}. The bounds in these works obtain the $5/12$ free-slip scaling for large slip lengths and large Prandtl numbers, but the no-slip $1/3$ scaling could not be established. For a discussion of the boundary conditions we refer to \cite{NetoEvansBonaccursoButtCraig2005}. Earlier estimates on the Nusselt number are reviewed in \cite{Nobili23}.

We finally introduce the mathematical model for the Rayleigh--B\'enard convection under consideration:
 We consider a  layer of fluid that is confined in a cell $[0,\ell]^2\times [0,1]$, where $\ell$ is the aspect ratio of the horizontal extension to the vertical extension. %For technical reasons, we assume that our layer is not too elongated,
%\[
%\ell \le \frac{5\pi}2.
%\]
The evolution of the fluid and the transport of heat inside this layer is the described by the following system of partial differential equations:
\begin{align}
\partial_t T + u\cdot \grad T & = \laplace T,\label{50}\\
\div u &=0,\label{51}\\
-\laplace u +\grad p & =  \Ra Te_3.\label{52}
\end{align}
Here, $T=T(t,x)\in \R$ is the temperature, $u=u(t,x)\in\R^3$ is the fluid velocity and $p=p(t,x)\in\R$ is the hydrodynamic pressure. The vector $e_3\in\R^3$ is the unit normal vector pointing in the direction of $x_3$, which is the upward direction. The first of these equations describes the transport of heat due to advection and conduction. Equations \eqref{51} and \eqref{52} are the Stokes equations, in which  buoyancy due to thermal expansion is included by a forcing term acting on an incompressible fluid.  This modelling ansatz is commonly referred to as the Boussinesq approximation.

For notational simplicity, in the following, we will write $x=(y,z)$ with $y\in[0,\ell]^2$ and $z\in[0,1]$, and $u = (v,w)\in \R^2\times \R$ to distinguish between horizontal and vertical components.

We suppose that all functions are $\ell$-periodic in both horizontal directions. Heating at the bottom  and cooling at the top boundaries are modelled by
\begin{equation}
\label{53}
T=1\quad\mbox{on }\{z=0\},\qquad T=0 \quad \mbox{on }\{z=1\}.
\end{equation}
The Navier-slip boundary conditions are given by
\begin{equation}
\label{54}
w=0,\, v= \sigma \partial_z v\quad\mbox{on }\{z=0\},\qquad w=0,\, v=-\sigma \partial_z v\quad \mbox{on }\{z=1\}.
\end{equation}
As already mentioned in the introduction, we see here that for vanishing slip length $\sigma=0$, we recover the no-slip boundary conditions $(v,w)=0$, while for infinite slip length, $\sigma=\infty$, we are concerned with free-slip boundary conditions $( \partial_z v,w)=0$. 

The Nusselt number is the ratio of the total upward heat flux to the purely conductive heat flux. With regard to our model \eqref{50}, this amounts to 
\[
\Nu = \int_0^1 \la (uT-\grad T)\cdot e_z\ra_{\ell}\,\dd z = \int_0^1 \la  wT  \ra_{\ell}\,\dd z + 1,
\]
where $\la \cdot \ra_{\ell}$ denotes the horizontal length and large-time average
\[
\la f\ra_{\ell} = \limsup_{t\to \infty} \frac1{t\ell^2} \int_0^t\int_{[0,\ell]^2} f\, \dd y \dd t.
\]

The remainder of the paper is devoted to the proof of Theorem \ref{T1}.

\section{Proofs}\label{S2}
We provide some notation. The \emph{characteristic function} of a (measurable) set $A$ will be denoted by $\chi_A$. The \emph{Fourier transform of a periodic function} $f:[0,L]^2\to \R$ is given by
\[
\widehat f(k) = \frac1{L^2}\int_{[0,L]^2} e^{-ik\cdot y}f(y)\dd y, 
\]
with wave number $k\in\frac{2\pi}L\Z^2$. Then, $f$ has the Fourier series representation
\[
f(y) =\sum_{k\in \frac{2\pi}{L}\Z^2} e^{ik\cdot y}\widehat f(k).
\]
The \emph{Fourier transform of a globally defined function} $f:\R^2\to \R$ is given by
\[
\widehat f(\xi) = \frac1{2\pi}\int_{\R^2} e^{-\xi\cdot y}f(y)\dd y.
\] 
The frequency $\xi$ takes values in $\R^2$.

\subsection{Reformulation of the problem and discussion of the Nusselt number}
Before outlining the strategy of the proof, we will reformulate the problem is a way that was introduced earlier in \cite{OttoSeis11}.

Following \cite{OttoSeis11}, we rescale all variables in such a way that the large dimensionless constant $\Ra$ drops out of the equations. More specifically,  rescaling length by $\Ra^{-1/3}$ and time by $\Ra^{-2/3}$, and setting
\[
H = \Ra^{1/3}, \quad L = \Ra^{1/3}\ell,
\]
we are concerned with the equations
\begin{align}
\partial_t T + u\cdot \grad T & = \laplace T,\label{18}\\
\div u &=0,\label{19}\\
-\laplace u +\grad p & =  Te_z.\label{20}
\end{align}
in the layer $[0,L]^2\times [0,H]$. The temperature boundary conditions now read
\begin{equation}\label{21}
T = 1 \quad \mbox{on }\{z=0\},\qquad T=0\quad \mbox{on }\{z=H\},
\end{equation}
and the Navier-slip boundary conditions are
\begin{equation}\label{22a}
w=0,\, v= \sigma \partial_z v\quad\mbox{on }\{z=0\},\qquad w=0,\, v=-\sigma \partial_z v\quad \mbox{on }\{z=H\}.
\end{equation}
The Nusselt number is the average upward heat flux
\begin{equation}
\label{36}
\Nu =\frac1{H} \int_0^H \la (uT-\grad T)\cdot e_z\ra_L\, \dd z   .
\end{equation}
 In the following, we will simply write $\la \cdot\ra = \la \cdot \ra_L$ for the horizontal length and large-time average.

In the new variables, our main result in Theorem \ref{T1} reads as follows.

\begin{theorem}
\label{T4}Let $\sigma>0$ be given. In the regime $H\gg1$, the Nusselt number satisfies the bounds
\[
\Nu \lesssim \begin{cases} 1 &\mbox{for }\sigma\lesssim 1,\\
\sigma^{1/2}&\mbox{for }1\lesssim \sigma \lesssim H^{1/2},\\
H^{1/4}& \mbox{for }\sigma \gtrsim H^{1/2}.
\end{cases}
\]
\end{theorem}

Our proof interpolates between the no-slip and Navier-slip regimes on the one hand, in which we will exploit and  develop further the ideas by Chanillo and Malchiodi, and the free-slip regime on the other hand, in which will use key estimates of Whitehead and Doering --- though presented in a completely new fashion. 
We will actually prove the following two estimates: First, by extending the results in \cite{ChanilloMalchiodi24},  we find that
\begin{equation}\label{15}
\Nu \lesssim 1+\sigma^{1/2},
\end{equation}
provided that $\sigma\le H$ and $H$ is large. Proving this estimate takes the main part of this paper. We stress that for large slip length $1\le \sigma\le H$, the scaling is new and the analysis is not  of perturbative nature. Instead, the leading order contributions in the Nusselt number have to be carefully identified and estimated.
Second, by invoking ideas from \cite{WhiteheadDoering11}, we obtain
\begin{equation}
\label{16}
\Nu \lesssim H^{1/4}.
\end{equation}
Apparently, the crossover from the Navier-slip  scaling in \eqref{15} to the free-slip scaling \eqref{16} occurs at slippage lengths of the order $\sigma\sim H^{1/2}$.

Starting point in either case is the observation from \cite{OttoSeis11} that the Nusselt  number can be localised near the boundary,
\begin{equation}\label{1}
\Nu \le \frac1{\delta}\int_0^{\delta} \la wT \ra\, \dd z +\frac1{\delta},
\end{equation}
where $\delta\in(0,1)$ is an arbitrary number. We  will eventually chose $\delta$ optimally in order to balance both terms in \eqref{1}. Physically, this optimal $\delta$ corresponds to the  size of the thermal boundary layers, in which the heat drop is essentially conductive \cite{Seis13}. The localisation strategy in \eqref{1} can be adapted to estimate the dissipation in a number of fluid problems including channel or pipe flows or further problems of thermal convection \cite{Seis15}.

We remark further    that because $w$ is mean-free by the incompressibility \eqref{19} and the no-flux boundary condition in \eqref{22}, we may always replace $T$ in the Nusselt number bound \eqref{1} by the mean-free temperature
\[
\theta = T - \avint_{[0,L]^2} T\,\dd y,
\]
so that $\la wT\ra =\la w\theta\ra$.

Our bounds on the convective term in \eqref{1} rely mostly on the analysis of the following fourth order elliptic  problem for the vertical velocity
\begin{equation}
\label{35}
\Delta^2 w = -\Delta_y \theta\quad \mbox{in }[0,L]^2\times [0,H],
\end{equation}
which can be derived from the Stokes equations \eqref{19}, \eqref{20} by differentiation. Using the incompressibility \eqref{20}, the boundary conditions
 \eqref{22a} can be rewritten  as
 \begin{equation}\label{22}
w=0,\, \partial_z w= \sigma \partial_z^2w\quad\mbox{on }\{z=0\},\qquad w=0,\, \partial_z w=-\sigma \partial_z^2 w\quad \mbox{on }\{z=H\}.
\end{equation}
Both \eqref{35} and \eqref{22} completely determine  $w$.  Moreover, we use the fact that $\theta$ satisfies 
\begin{equation}\label{17}
\|\theta\|_{L^{\infty}}\le 1
\end{equation}
by the maximum principle for the advection-diffusion equation \eqref{18} (at least if this bound is satisfied by the initial datum), and the Nusselt number can be expressed in terms of the thermal dissipation,
\begin{equation}
\label{23}
\Nu  = \int_0^H \la |\grad T|^2 \ra\, \dd z.
\end{equation}
The latter can be seen by simply testing the advection-diffusion equation \eqref{18} with the temperature $T$, and using the fact the Nusselt number is constant on every horizontal slice
\[
\Nu = \la w\theta - \partial_z T\ra (z),
\]
for any $z\in[0,H]$, so that $\Nu = -\la \partial_z|_{z=0} T\ra$ in particular. We remark that the localisation in \eqref{1} is a consequence of the previous observation and both the maximum principle \eqref{17} and the nonnegativity of the temperature.

We start addressing the bound in \eqref{15}.

\subsection{Optimal bound in the  no-slip and Navier-slip regimes}

In order to stress the dependence on the slip length, which is crucial in the Navier-slip regime, we write $w_{\sigma}=w$ in the following. 

Our aim is to compare the actual velocity with the no-slip velocity field that approximates $w_{\sigma}$ in the regime $\sigma\ll1$, that is, we consider
\begin{equation}\label{7}
\laplace^2 w_0 = -\laplace_y \theta
\end{equation}
in the layer $[0,L]^2\times [0,H]$, and we suppose that $w_0$ has the no-slip boundary conditions
\begin{equation}\label{8}
w_0=\partial_z w_0=0\quad\mbox{on }\{z=0,H\}.
\end{equation}
The main result by Chanillo and Malchiodi \cite{ChanilloMalchiodi24} shows that the associated flux quantity is uniformly bounded in any boundary layer of order-one thickness. More precisely, they show that for any $\theta$ satisfying the maximum principle \eqref{17} and for any $w_0$ solving the inhomogeneous boundary value problem \eqref{7}, \eqref{8}, it holds that
\begin{equation}
\label{24}
|\la w_0\theta\ra|\lesssim 1\quad\mbox{for any }z\in [0,1],
\end{equation}
provided that $H$ is sufficiently large. Exploiting this observation, it is enough to further bound the correction in 
 \begin{equation}\label{29}
\Nu  \lesssim \left| \frac1{\delta}\int_0^{\delta}\la (w_{\sigma}-w_0)\theta\ra\, \dd z \right|+ \frac1{\delta}.
\end{equation}

Inspired by the analysis in \cite{ChanilloMalchiodi24}, we will decompose the limiting vertical velocity into upper and lower contributions: We let $v_0$ denote the solution to the truncated problem
\begin{equation}\label{40}
\laplace^2 v_0 = -\chi_{[0,\frac{H}2]}\laplace_y\theta,
\end{equation}
together with the Dirichlet and Neumann boundary conditions in \eqref{8}.
We then   decompose the correction function into  
\begin{equation}\label{42}
w_{\sigma}-w_0 = h_{\sigma} + g_{\sigma},
\end{equation}
where $h_{\sigma}$ solves the bi-Laplace problem
\begin{equation}
\label{2}
\Delta^2 h_{\sigma}=0
\end{equation}
inside of the domain, and the boundary conditions
\begin{equation}
\label{31}
h_{\sigma}=0,\, \partial_z h_{\sigma} = \mp\sigma\left( \partial_z^2 h_{\sigma} + \partial_z^2 v_0\right)\quad \mbox{on }\left\{z=\frac{H}2\pm\frac{H}2\right\},
\end{equation}
where the minus sign occurs at the upper boundary, and $g_{\sigma}$ solves the analogous problem for the upper half, that is
\begin{equation}
\label{32b}
\Delta^2 g_{\sigma}=0
\end{equation}
inside of the domain, and the boundary conditions
\begin{equation}
\label{3}
g_{\sigma}=0,\, \partial_z g_{\sigma} = \mp\sigma\left( \partial_z^2 g_{\sigma} + \partial_z^2 (w_0-v_0)\right)\mbox{on }\left\{z=\frac{H}2\pm\frac{H}2\right\}.
\end{equation}

\medskip

\noindent
\textbf{The correction term generated by $h_{\sigma}$.}
We consider the associated problem on the half-space: Denoting by $\tilde \theta$ the   truncation of $\theta$, 
\[
\tilde \theta (x) = \chi_{[0,\frac{H}2]}(z)\theta(x),
\]
and extending this function periodically in $y$, we consider 
\[
\tilde h_{\sigma}(x) = \int_{\R^3_+}B_{\sigma}(x,\tilde x)\tilde\theta(\tilde x)\,\dd \tilde x,
\]
where the kernel $ B_{\sigma}(x,\tilde x)$ solves the boundary value problem
\begin{align*}
\Delta_x^2 B_{\sigma}(\cdot ,\tilde x) &=0\quad\mbox{in }\R^3_+,\\
B_{\sigma}(\cdot,\tilde x)  &=0\quad\mbox{on }\partial \R^3_+,\\
\partial_z B_{\sigma}(\cdot,\tilde x) &= \sigma\left(\partial_z^2 B_{\sigma}(\cdot ,\tilde x)+\partial_z^2 K_0(\cdot ,\tilde x\right))\quad\mbox{on }\partial\R^3_+.
\end{align*}
Here $K_0(x,\tilde x)$ is the kernel associated to $w_0$, extended to the half space. It was established in \cite{ChanilloMalchiodi24} that this kernel is related to the Poisson kernel 
\[
P_s(y) = \frac1{2\pi} \frac{s}{(s^2+|y|^2)^{\frac32}}.
\]
 More specifically, with regard to Remark 2.7 in \cite{ChanilloMalchiodi24}, we have that
\begin{equation}\label{60}
\partial_z^2 K_0(x ,\tilde x) = \tilde z\Delta_{y} P_{\tilde z}(y-\tilde y)\quad \mbox{for any }x=(y,0)\in \partial \R^3_+.
\end{equation}
Thanks to the symmetry in horizontal direction, we have and write (slightly abusing notation),
\[
B_{\sigma}(x,\tilde x) = B_{\sigma}((y-\tilde y,z),(0,z)) = B_{\sigma}(y-\tilde y,z,\tilde z).
\]

Horizontally Fourier transforming the kernel problem, we obtain an ODE for the Fourier transform $
\widehat{B}_{\sigma}(\xi,z,\tilde z) $
that can be solved explicitly: Using that the  Fourier transform of the Poisson kernel is well-known, $\widehat{P}_{s}(\xi) = \frac1{2\pi} e^{-s|\xi|}$, see, for instance Exercise 2.2.11 in \cite{Grafakos14}, we find
\begin{align*}
\widehat{B}_{\sigma} (\xi,z,\tilde z) = - \frac{\sigma z\tilde z}{2\pi}\frac{|\xi|^2}{1+2\sigma |\xi|} e^{-(z+\tilde z)|\xi|} 
 = -\frac{\sigma z\tilde z}{2\pi} |\xi|^2 e^{-(z+\tilde z)|\xi|} p_{\sigma}(\xi),
\end{align*}
where $p_{\sigma}(\xi) = (1+2\sigma|\xi|)^{-1}$ is an order-zero symbol. Denoting by $D_{\sigma}=p_{\sigma}(\grad_y)$ the associated pseudo-differential operator,
transformation back into physical variables gives
\[
B_{\sigma} (y,z,\tilde z)  =\sigma z\tilde z D_{\sigma}\laplace_{ y} P_{z+\tilde z}(y)  ,
\]
so that
\[
\tilde h_{\sigma}(x) = \sigma z\int_0^{\infty}\int_{\R^2} \tilde z \Delta_y P_{z+\tilde z}(y-\tilde y) (D_{\sigma}\tilde \theta)(\tilde y,\tilde z)\,\dd \tilde y\dd\tilde z.
\]

Next, we derive decay estimates for the kernel $B_{\sigma}$.
\begin{lemma}\label{L8}
It holds that
\begin{align*}
|\grad_y(-\Delta_y)^{-1} B_{\sigma}(y,z,\tilde z)& \lesssim \frac{\sigma z\tilde z}{(z+\tilde z+| y|)^3},\\
|B_{\sigma}(y,z,\tilde z)|  & \lesssim \frac{\sigma z\tilde z}{(z+\tilde z+| y|)^4},\\
|\grad_y B_{\sigma}(y,z,\tilde z)|+|\partial_z B_{\sigma}(y,z,\tilde z)| &\lesssim \frac{\sigma \tilde z}{(z+\tilde z+| y|)^4},\\
|\Delta_y B_{\sigma}(y,z,\tilde z)|+|\partial_z^2 B_{\sigma}(y,z,\tilde z)| &\lesssim \frac{\sigma \tilde z}{(z+\tilde z+| y|)^5}.
\end{align*}
\end{lemma}

\begin{proof}
We establish all estimates simultaneously by considering the multipliers
\[
m^j(\xi) =  \frac1{2\pi}\frac{|\xi|^{j+1}}{1+2\sigma |\xi|} e^{-(z+\tilde z)|\xi|}
\]
for $j\in\N_0$, so that 
\begin{gather*}
\widehat{\grad_y(-\Delta_y)^{-1} B}_{\sigma}(\xi)=  - \sigma z\tilde z   m^0(\xi),\quad 
 \widehat{B}_{\sigma} (\xi) =  - \sigma z\tilde z   m^1(\xi) ,\\
   \partial_z \widehat{B}_{\sigma}(\xi) =-\sigma \tilde z m^1(\xi) +\sigma z\tilde zm^2(\xi),\quad \partial_z^2 \widehat{B}_{\sigma} (\xi)=2\sigma \tilde z m^2(\xi) - \sigma z\tilde z m^3(\xi),\\
   \widehat{\grad_y B}_{\sigma} (\xi)= -i\sigma z\tilde z \xi m^1(\xi),\quad \Delta_y \widehat{B}_{\sigma}(\xi) =\sigma z \tilde z m^3(\xi).
  \end{gather*}
We have to study the multipliers.

First, we notice that
\begin{align*}
|\widecheck{m}^j(y)| &=\frac{1 }{(2\pi)^2} \left| \int_{\R^2}e^{i\xi\cdot y} \frac{|\xi|^{j+1}}{1+2\sigma|\xi|} e^{-|\xi|(z+\tilde z)}\, \dd \xi\right|\\
&\lesssim    \int |\xi|^{j+1} e^{-|\xi|(z+\tilde z)}\, \dd \xi\\
&\lesssim \frac{1}{(z+\tilde z)^{j+3}}.
\end{align*}

For  $|y|\gg z+\tilde z$, we get better bounds by applying a Mikhlin--H\"ormander-type argument: We localise the Fourier multiplier $m^j$ in Fourier space by setting
\[ 
m_k^j(\xi) = \psi(2^{-k}\xi)m^j(\xi),
\]
where $\psi$ is a nonnegative cut-off function supported on the annulus $1/2\le |\xi|\le 2$ and generating a dyadic partition of unity,
\[
\sum_{k=-\infty}^{\infty} \psi(2^{-k}\xi)=1,
\]
for any $\xi\not=0$. 
It is readily checked that
\[
|\partial_{\xi}^{\gamma} m_k^j(\xi)| \lesssim    2^{(j+1-|\gamma|)k} ,
\]
for any multi-index $\gamma\in\N_0^2$. In particular, we have the bound
\[
\|\partial_{\xi}^{\gamma} m_k^j \|_{L^1(\R^2)} \lesssim  2^{(j+3-|\gamma|)k} ,
\]
in view of the support of the multiplier. Therefore, using elementary properties of the Fourier transform, we find for $K^j_k = \widecheck{m}_k^j$ that
\[
|y^{\gamma} K_k^j(y)|  = |\widecheck{\partial_{\xi}^{\gamma} m_k^j}(y)| \le \|\partial_{\xi}^{\gamma} m_k^j\|_{L^1(\R^2)} \lesssim 2^{(j+3-|\gamma|)k}.
\]
In particular, for $m=|\gamma|$, this estimate leads to 
\[
|y|^m |K_k^j(y)| \lesssim  2^{(j+3-m)k}.
\]
Summation  over $k$, using that $\widecheck{m^j}(y) = \sum_{k} K_k^j(y)$ and choosing $m=0$ for small values of $2^k|y|$ and $m=j+4$ for large  values of $2^k|y|$ eventually gives
\begin{align*}
|\widecheck{m^j}(y)|\le \sum_{k=-\infty}^{\infty}|K_k^j(y)| & \le \sum_{2^k\le \frac1{|y|}}|  K_k^j(y)|+\sum_{2^k\ge \frac1{|y|}}|  K_k^j(y)|\\
&\lesssim  \sum_{2^k\le \frac1{|y|}} 2^{(j+3)k }  + \frac{1}{|y|^{j+4}} \sum_{2^k\ge \frac1{|y|}}   2^{-k}\lesssim \frac{1}{|y|^{j+3}}.
\end{align*}

We combine both estimates to deduce 
\[
|\widecheck{m^j}(y)|\lesssim \frac1{(z+\tilde z +|y|)^{j+3}}.
\]
The bounds on $B_{\sigma}$, $\partial_z B_{\sigma}$, $\partial_z^2 B_{\sigma}$ and $\Delta_y B_{\sigma}$ now follow immediately. The remaining estimates for the terms that involve horizontal gradients are obtained analogously.
\end{proof}

 For later reference, we derive the following estimates from the previous bound on the kernel.
 \begin{lemma}\label{L13}
The following estimates are true:
\[
\sup_{[H-1,H]} \la |\grad  \tilde h_{\sigma}|^2\ra \lesssim  \sigma^2, \quad \sup_{[H-1,H]} \la |\grad_y^2 \tilde h_{\sigma}|^2\ra \lesssim \frac{\sigma^2}{H^2}.
\]
\end{lemma}

\begin{proof}
We provide the argument exemplified in the first order vertical derivative term. The bounds on all other derivatives  are obtained similarly. Actually, we derive a  stronger pointwise bound. Using the kernel estimates from Lemma \ref{L8}, the temperature bound \eqref{17} and the definition of the truncation, we have that
\begin{align*}
|\grad  \tilde h_{\sigma} (x)|& \le \int_0^{H} \int_{\R^2} |\grad_y B_{\sigma} (\tilde y,z,\tilde z)|+|\partial_z B_{\sigma} (\tilde y,z,\tilde z)|\, \dd \tilde y\dd \tilde z\\
&\lesssim \sigma \int_0^H \int_{B_{z+\tilde z}(0)} \frac{\tilde z}{(z+\tilde z)^4}\,\dd\tilde y\dd \tilde z + \sigma \int_0^H \int_{\R^2\setminus B_{z+\tilde z}(0)} \frac{\tilde z}{|\tilde y|^4}\,\dd\tilde y\dd \tilde z\\
&\lesssim \sigma \int_0^H \frac{\tilde z}{(z+\tilde z)^2}\, \dd \tilde z \lesssim \frac{\sigma}{H^2}\int_0^H\tilde z\, \dd\tilde z \lesssim \sigma,
\end{align*}
where we have used that $z\gtrsim H$.

For the second-order horizontal derivatives proceeds similarly, we use  the observation that $\la |\grad_y^2 \tilde h_{\sigma}|^2\ra  = \la (\Delta_y \tilde h_{\sigma})^2\ra$, as can be seen via an integration by parts. Then, the last estimate in Lemma \ref{L8} becomes available and we can argue very similarly to the previous bound.
\end{proof}

Estimating the singular part of $\tilde h_{\sigma}$ is rather subtle. Our argument follows the one of Proposition 3.2 in \cite{ChanilloMalchiodi24}, which has to be adapted to our setting.

\begin{prop}
\label{P2}
For $H\gg1$, there is the estimate
\[
\left|\avint_{[-\frac{L}2,\frac{L}2]^2}   \tilde h_{\sigma}\theta \,\dd y\right| \lesssim \sigma z,
\]
for any $z\in [0,1]$. In particular, it is true that
\[
\left|\frac1{\delta} \int_0^{\delta}\la \tilde h_{\sigma} \theta\ra\, \dd z\right| \lesssim \sigma \delta,
\]
for any $\delta\in(0,1]$.
\end{prop}

\begin{proof}Our argument analyses the elliptic problem defining $\tilde h_{\sigma}$. We can thus neglect the time dependence in the following.

We start by introducing the horizontally truncated functions
\[
\zeta = \chi_{[-\frac{L}2,\frac{L}2]^2}  \theta,\quad \tilde \zeta = \chi_{[-L,L]^2}\tilde \theta.
\]
The estimates on the kernels in Lemma \ref{L8} can be  exploited to split off a non-singular part from the extended correction function near the bottom boundary. More specifically, we write
\begin{align*}
\tilde h_{\sigma}(x)& = \sigma z \int_0^{\infty} \int_{\R^2} \tilde z\Delta_yP_{z+\tilde z}(y-\tilde y) D_{\sigma}\tilde \zeta\,\dd\tilde y\dd\tilde z \\
&\quad +\int_0^{\infty}\int_{\R^2\setminus [-L,L]^2} B_{\sigma}(y-\tilde y,z,\tilde z)\tilde \theta(\tilde y,\tilde z)\, \dd \tilde y\dd \tilde z.
\end{align*}
For the second term, we estimate with the help of the bounds from Lemma \ref{L8} and the maximum principle \eqref{17} for the temperature for any $y\in[-\frac{L}2,\frac{L}2]^2$
\begin{align*}
\left|\int_0^{\infty}\int_{\R^2\setminus [-L,L]^2} B_{\sigma}(y-\tilde y,z,\tilde z)\tilde \theta(\tilde y,\tilde z)\, \dd \tilde y\dd \tilde z\right|
&\lesssim \int_0^H \int_{\R^2\setminus[-\frac{L}2,\frac{L}2]^2} \frac{\sigma z\tilde z}{(z+\tilde z+|\tilde y|)^4}\,\dd\tilde y \\
&\lesssim \sigma z\frac{H}{L}.
\end{align*}
Because $H\sim L$ (in the sense that both have the same scaling behaviour in terms of $\Ra$), using the maximum principle for the temperature again, we then have
\begin{align*}
\avint_{[-\frac{L}2,\frac{L}2]^2}\theta \tilde h_{\sigma}\, \dd y  =\frac{\sigma z}{L^2}\int_{0}^H \tilde z \int_{\R^2}\int_{\R^2} \zeta(y,z)\Delta_y P_{z+\tilde z}(y-\tilde y) (D_{\sigma}\tilde \zeta)(\tilde y,\tilde z)\,\dd \tilde y \dd y\dd \tilde z + O(\sigma z).
\end{align*}

We now use the semi-group property of the Poisson kernel, $P_s = P_{\frac{s}2}\ast P_{\frac{s}2}$, its radial symmetry and an integration by parts to rewrite this identity as
\begin{align*}
\mel \avint_{[-\frac{L}2,\frac{L}2]^2}\theta \tilde h_{\sigma}\, \dd y \\
& =\frac{\sigma z}{L^2}\int_{0}^H \tilde z  \int_{\R^2} \left(\grad_y P_{\frac{z+\tilde z}2} \ast \zeta(\cdot,z)\right)\cdot \left(\grad_y P_{\frac{z+\tilde z}2} \ast (D_{\sigma}\tilde \zeta)(\cdot,\tilde z)\right) \dd y\dd \tilde z +O(\sigma z).
\end{align*}
We may once more split off a term that can be controlled by simple kernel estimates. For this purpose, we decompose the integral on the right-hand side 
\begin{align*}
\avint_{[-\frac{L}2,\frac{L}2]^2}\theta \tilde h_{\sigma}\, \dd y 
 & =  \mathrm{I}(z) + \mathrm{II} (z) +O(\sigma z),
\end{align*}
where we have introduced
\begin{align*}
\mathrm{I}(z) & = \frac{\sigma z}{L^2}\int_{0}^H \tilde z \int_{\R^2\setminus [-2L,2L]^2} \left(\grad_y P_{\frac{z+\tilde z}2} \ast \zeta(\cdot,z)\right)\cdot \left(D_{\sigma}\grad_y P_{\frac{z+\tilde z}2} \ast  \tilde \zeta (\cdot,\tilde z)\right) \dd y \dd \tilde z,\\
\mathrm{II}(z)& =  \frac{\sigma z}{L^2}\int_{0}^H \tilde z \int_{ [-2L,2L]^2} \left(\grad_y P_{\frac{z+\tilde z}2} \ast \zeta(\cdot,z)\right)\cdot \left(\grad_y P_{\frac{z+\tilde z}2} \ast (D_{\sigma}\tilde \zeta)(\cdot,\tilde z)\right)\dd y \dd \tilde z .
\end{align*}

To estimate the first term, we make use of the fact that the derivative of the Poisson kernel satisfies the bound
\begin{equation}
\label{26}
|\grad_y P_{\frac{z+\tilde z}2}(y)|+  |\grad_y D_{\sigma}P_{\frac{z+\tilde z}2}(y)| \lesssim \frac1{(s^2 +|y|^2)^{3/2}},
\end{equation}
as can be straightforwardly verified or follows from Lemma \ref{L8}, respectively, because $  D_{\sigma}\grad_y P_{\frac{z+\tilde z}2 }(y)= 4(\sigma z\tilde z)^{-1} \grad_y(-\Delta_y)^{-1} B_{\sigma}(y,z/2,\tilde z/2)$. For instance, using the temperature bound \eqref{17}, we have for any $y\not\in[-2L,2L]^2$ that
\[
\left|\left(\grad_y P_{\frac{z+\tilde z}2}\ast \zeta(\cdot,z)\right)(y)\right| \lesssim \int_{[-\frac{L}2,\frac{L}2]^2} \frac1{(|y-\tilde y|^2+\tilde z^2)^{\frac32}} \, \dd \tilde y \lesssim \frac{L^2}{(|y|+\tilde z)^3}.
\]
The same estimate applies to the analogous expression with $\tilde \zeta$. We use this information to derive that
\begin{equation}
\label{28}
\begin{aligned}
|\mathrm{I}(z)| & \lesssim \sigma z L^2 \int_{0}^H \int_{\R^2\setminus [-2L,2L]^2} \frac{ \tilde z}{(|y|+\tilde z)^6} \,\dd y\dd \tilde z \\
& \lesssim \sigma z L^2 \int_0^{\infty}\frac{\tilde z}{(L+\tilde z)^3}\,\dd \tilde z\int_{\R^2\setminus [-2L,2L]^2} \frac1{|y|^3} \, \dd y \lesssim  \sigma z.
\end{aligned}
\end{equation}

We turn to the estimate of $\mathrm{II}$. Via Jensen's inequality, we obtain
\[
|\mathrm{II}(z)| \lesssim \frac{\sigma z}{L} \left(\int_{\R^2} \left(\int_{0}^H \left(\grad_y P_{\frac{z+\tilde z}2} \ast \zeta(\cdot,z)\right) \cdot \left(\grad_y P_{\frac{z+\tilde z}2} \ast D_{\sigma}\tilde \zeta(\cdot,\tilde z)\right)  \tilde z\, \dd\tilde z\right)^2\dd  y\right)^{1/2}.
\]
Performing a change of variables $s=(z+\tilde z)/2$ and defining
\begin{gather*}
 \Psi(s,y) = s\grad_y P_s(y),\quad F_z(s,y) = D_{\sigma}\tilde \zeta(y ,2s-z),\\
  \rho_z(y) =\zeta(y,z),\quad m_z(s) = \chi_{(\frac{z}2,\frac{z+H}2)}(s) \frac{2s-z}{s},
\end{gather*}
the latter can be rewritten as
\[
|\mathrm{II}(z)| \lesssim \frac{\sigma z}{L} \left(\int_{\R^2} \left(\int_0^{\infty} (\Psi\ast\rho_z)(s,y)\cdot (\Psi\ast F_z)(s,y) m_z(s)\frac{\dd s}{s}\right)^2\dd y\right)^{1/2}.
\]
Noticing that $|m_z(s)|\lesssim 1$, and using a variant of a bilinear estimate by Coifman and Meyer \cite{CoifmanMeyer78}, see Proposition 2.3 of \cite{ChanilloMalchiodi24}, we estimate this expression further by
\begin{equation}\label{27}
| \mathrm{II}| \lesssim \frac{\sigma z}{L} \|\rho_z\|_{\text{BMO}(\R^2)} \|N(\Psi\ast F_z)\|_{L^2(\R^2)},
\end{equation}
where $N$ is the non-tangential maximal function defined as 
\[
N g(y)  = \sup\left\{ |g(\hat y,z)|:\: |y-\hat y| \le z\right\}.
\]
By the temperature bound \eqref{17} and the definition of $\rho$, the first norm on the right-hand side is trivially controlled,
\[
\|\rho_z\|_{\text{BMO}(\R^2)} \le 2 \|\theta(\cdot,z)\|_{L^{\infty}(\R^2)}\lesssim1.
\]
To estimate the second norm,   we have thanks to the  bound on the second term in \eqref{26} and the temperature maximum \eqref{17}  that
\[
|(\Psi\ast F_z)(s,y)| \lesssim \int_{[-L,L]^2} \frac{s}{(| y-\tilde y| + s)^3} \, \dd\tilde y \lesssim \int_{\R^2} \frac{s}{(| \hat y| + s)^3} \, \dd\hat y\lesssim 1.
\]
This estimate can be improved for   $y \not\in [-2L,2L]^2$, because then   $|y-\tilde y|\gtrsim L\gtrsim |y|$. Therefore, for any  $y \not\in [-2L,2L]^2$
\[
|(\Psi\ast F_z)(s,y)| \lesssim  \int_{[-L,L]^2} \frac{s}{(|  y| + s)^3} \, \dd\tilde y \lesssim \frac{L^2}{|y|^2}.
\]
Combining both estimates, we find
\begin{align*}
\int_{\R^2} |N(\Psi\ast F_z)|^2\,\dd y & = \int_{[-2L,2L]^2} |N(\Psi\ast F_z)|^2\,\dd y +\int_{\R^2\setminus [-2L,2L]^2} |N(\Psi\ast F_z)|^2\,\dd y
\\
&\lesssim L^2 + L^4 \int_{\R^2\setminus [-2L,2L]^2}\frac1{|y|^4}\, \dd y \lesssim L^2.
\end{align*}
We plug all the gathered information into \eqref{27} to conclude that
\[
|\mathrm{II}(z)| \lesssim \sigma z.
\]
\end{proof}

Next, we have to estimate the error that is due to studying the problem on the half-space instead of the bounded domain problem. That is, we have to estimate the error   $h_{\sigma}-\tilde h_{\sigma}$.
First, we consider the solution $f_{\sigma}$ to the problem
\[
\Delta^2 f_{\sigma} = 0
\]
in the layer $[0,L]^2\times [0,H]$ with   boundary conditions  
\begin{align*}
f_{\sigma}=0,\,  \partial_z f_{\sigma}-\sigma \partial_z^2 f_{\sigma} &=  \sigma   \partial_z^2 v_0 - \sigma \partial_z^2 \tilde v_0\quad\mbox{on }\{z=0\},\\
f_{\sigma}=0,\,   \partial_z  f_{\sigma} +\sigma\partial_z^2   f_{\sigma} &=- \partial_z \tilde h_{\sigma}-\sigma\partial_z^2 \tilde h_{\sigma} -\sigma \partial_z^2 v_0 \quad\mbox{on }\{z=H\},
\end{align*}
where $\tilde v_0$ denotes the solution to the half-space problem associated to $\tilde \theta$, so that
\[
\partial_z^2 \tilde v_0(x) = \int_{\R^3_+} \tilde z \Delta_y P_{\tilde z}(y-\tilde y) \tilde \theta(\tilde y,\tilde z)\,\dd (\tilde y,\tilde z),
\]
cf.~\eqref{60}.

 We have  the following error estimate.

\begin{lemma}
\label{L10}
The following estimate is true:
\begin{align*}
 \left.\la (\partial_z^2 v_0-\partial_z^2 \tilde v_0)^2\ra\right|_{z=0}+ \left. \la (\partial_z^2 v_0)^2\ra \right|_{z=H} \lesssim   1.
 \end{align*}
\end{lemma}

\begin{proof}We horizontally Fourier transform the problem for the difference $V_0=v_0-\tilde v_0$ in the bounded domain. The resulting ODE can be solved explicitly and yields that
\begin{align*}
\mel \partial_z^2 \widehat{V_0}(k,z)\\
& = \frac{A_k(z)e^{-|k|(3H+z)} + B_k(z)e^{-|k|(H-z)} +D_k(z) e^{-|k|(H+z)}+ E_k(z)e^{-|k|(3H-z)}  }{1 - 2e^{-2H|k|} (1+2H^2|k|^2) +e^{-4H|k|}},
\end{align*}
where
\begin{align*}
A_k(z)& = (1+(H-z)|k|) H^2|k|^2 a_k + (2+(H-z)|k|)H|k| b_k,\\
B_k(z) & = (1-(H-z)|k|) H^2 |k|^2 a_k  - (2-(H-z)|k|)H|k| b_k,\\
C_k(z) &  = -(1-z|k| + H|k|(3-2z|k|))H^2|k|^2 a_k\\
&\quad  -(2-z|k|-H|k|(3-2z|k|)) H|k|b_k,\\
D_k(z) &  = -(1+z|k| - H|k|(3+2z|k|))H^2|k|^2 a_k\\
&\quad  +(2+z|k|+H|k|(3+2z|k|)) H|k|b_k,
\end{align*}
and where $a_k = -  H^{-2}\widehat{\tilde v}_0(k,H)$ and $b_k = -H^{-1}\partial_z \widehat{\tilde v}_0(k,H)$. Notice that the denominator is bounded away from zero: Its minimal value is attained for $H|k|=2\pi H/L\sim1$ in the sense that both $H$ and $L$ scale identically in $\Ra$.

At the bottom boundary, we find that
\begin{align*}
\mel \partial_z^2 \widehat{V_0}(k,0)\\
& = \frac{4(a_k H^3|k|^3  +b_k H|k|)(e^{-H|k|}-e^{-3H|k|}) - 4b_k H^2 |k|^2(e^{-H|k|}+e^{-3H|k|})  }{1 - 2e^{-2H|k|} (1+2H^2|k|^2) +e^{-4H|k|}},
\end{align*}
It was proved in Proposition 3.4 of \cite{ChanilloMalchiodi24} that 
\begin{equation}\label{25}
H^{-2}|{\tilde v}_0 | + H^{-1}|\grad {\tilde v}_0| + |\grad^2 \tilde v_0|\lesssim 1\quad\mbox{at $z=H$},
\end{equation}
 and thus, in particular, in view of  Plancherel's identity, it holds that
\[
\sum_{k\in \frac{2\pi}{L}\Z^2}\left( |a_k|^2 +|b_k|^2\right)  = \avint_{[0,L]^2} H^{-4} \tilde v_0^2 + H^{-2}(\partial_z \tilde v_0)^2 \lesssim 1.
\]
Therefore, we may now brutally estimate
\[
|\partial_z^2 \widehat{V_0}(k,0)| \lesssim |a_k|+|b_k|,
\]
which is valid under the assumption that the denominator is not vanishing, and we deduce via Plancherel's identity, that 
\[
\left.\la (\partial_z^2 V_0)^2\ra\right|_{z=0}\lesssim 1,
\]
as stated in the lemma. 

Similarly, at the top boundary, we have that
\begin{align*}
\partial_z^2 \widehat V_0(k,H)
& = H|k|\left(2-8 H|k|e^{-2H|k|}-2e^{-4H|k|}\right) b_k\\
&\quad  - H^2|k|^2 \left(1+(2+ 4 H^2 |k|^2)e^{-2H|k|} - e^{-4H|k|}\right)a_k,
\end{align*}
and thus, we have
\[
|\partial_z^2 \widehat{V}_0(k,H)| \lesssim H|k| |b_k| + H^2 |k|^2 |a_k|.
\]
Via Plancherel's identity and using the definition of the coefficients $a_k$ and $b_k$, the latter implies that
\[
\left.\la (\partial_z^2 V_0)^2\ra\right|_{z=H} \lesssim \left.\la |\partial_z \grad_y \tilde v_0|^2\ra\right|_{z=H} +\left. \la |\grad^2_y\tilde v_0|^2\ra\right|_{z=H}.
\]
From the definition of $V_0$ and the bounds in \eqref{25}, we deduce that
\[
\left.\la (\partial_z^2 v_0)^2\ra\right|_{z=H} \lesssim \left.\la |\grad^2  \tilde v_0|^2\ra\right|_{z=H} \lesssim 1,
\]
which is the second estimate we aimed to prove.
\end{proof}

The error estimate for the vertical velocity and the earlier estimates on the correction term $\tilde h_{\sigma}$ in Lemma \ref{L13} allow us to suitably bound $f_{\sigma}$.

\begin{lemma}\label{L11}
Suppose that $\sigma\le H$. Then following estimate is true:
\[
\int_0^H \la |\grad^2 f_{\sigma}|^2\ra \, \dd z + \frac1{\sigma}\left.\la (\partial_z f_{\sigma})^2\ra\right|_{z=0,H} \lesssim \sigma.
\]\
\end{lemma}

\begin{proof}
Testing the bi-Laplace equation with $f_{\sigma}$ and integrating by parts yields
\[
\int_0^H \la |\grad^2 f_{\sigma}|^2\ra\, \dd z - \left.\la \partial_z f_{\sigma}\partial_z^2 f_{\sigma}\ra\right|_{z=0}^{z=H} = 0.
\]
We use now the Navier-slip boundary conditions to write
\begin{align*}
\mel \int_0^H \la |\grad^2 f_{\sigma}|^2\ra\, \dd z +\frac1{\sigma} \left.\la (\partial_z f_{\sigma})^2 \ra\right|_{z=0,H}\\
 & = \left.\la \partial_z f_{\sigma}(\partial_z^2 v_0-\partial_z^2\tilde v_0)\ra\right|_{z=0}\\
&\quad - \frac1{\sigma}\left.\la \partial_z f_{\sigma}\partial_z \tilde h_{\sigma}\ra\right|_{z=H} -  \left. \la \partial_z f_{\sigma} \partial_z^2 \tilde h_{\sigma}\ra\right|_{z=H} -\left. \la\partial_z f_{\sigma} \partial_z^2 v_0\ra\right|_{z=H}.
\end{align*}
With the help of Young's inequality and  the error bounds in Lemmas \ref{L13} and \ref{L10}, we then estimate
\begin{align*}
 \int_0^H \la |\grad^2 f_{\sigma}|^2\ra\, \dd z +\frac1{\sigma} \left.\la (\partial_z f_{\sigma})^2 \ra\right|_{z=0,H} & \lesssim \sigma  \left.\la (\partial_z^2 v_0-\partial_z^2\tilde v_0)^2 \ra\right|_{z=0}+  \sigma  \left.\la (\partial_z^2 v_0)^2 \ra\right|_{z=H}\\
&\quad +\frac1{\sigma}  \left.\la (\partial_z\tilde h_{\sigma})^2 \ra\right|_{z=H} +  \sigma  \left.\la (\partial_z^2 \tilde h_{\sigma})^2 \ra\right|_{z=H} \\
&\lesssim \sigma  + \frac{\sigma^3}{H^2}  \lesssim \sigma,
\end{align*}
because $\sigma \le H$. This is what we aimed to prove.
\end{proof}

Next, we show that this is good enough to control the associated part in the Nusselt number estimate.

\begin{lemma}\label{L12}
Suppose that $\sigma \le H$. Then following estimate is true:
\[
\left| \frac1{\delta} \int_0^{\delta}\la f_{\sigma}\theta\ra\, \dd z\right| +  \left| \frac1{\delta} \int_{H-\delta}^{H}\la f_{\sigma}\theta\ra\, \dd z\right|\lesssim\delta^{\frac32}\sigma^{\frac12}  + \delta \sigma .
\]
\end{lemma}

\begin{proof}
We prove the estimate near the bottom boundary. The one on the top boundary is obtained analogously. We use the maximum principle for the temperature \eqref{17} and the Dirichlet boundary conditions for $f_{\sigma}$ to estimate with the help of the Poincar\'e inequality
\[
\left| \frac1{\delta} \int_0^{\delta}\la f_{\sigma}\theta\ra\, \dd z\right| \le  \frac1{\delta} \int_0^{\delta}\la |f_{\sigma}|\ra\, \dd z \le  \int_0^{\delta} \la |\partial_z f_{\sigma}|\ra\, \dd z.
\]
Smuggling in the boundary term, applying once more the  Poincar\'e inequality and using Jensen's inequality together with the estimates from Lemma \ref{L11} then gives
\begin{align*}
\left| \frac1{\delta} \int_0^{\delta}\la f_{\sigma}\theta\ra\, \dd z\right| &\le  \delta  \int_0^{\delta} \la |\partial_z^2 f_{\sigma} |\ra\, \dd z + \delta \la | \partial_z f_{\sigma}|_{z=0}|\ra \\
&\le \delta^{3/2}\left(\int_0^H \la (\partial_z^2 f_{\sigma})^2\ra\, \dd z\right)^{1/2} +\delta \la ( \partial_z f_{\sigma})^2|_{z=0} \ra^{1/2} \\
&\lesssim \delta^{\frac32}\sigma^{\frac12}  + \delta \sigma ,
\end{align*}
which is our desired result.
\end{proof}

It remains to study the remainder term
\[
\tilde f_{\sigma} = h_{\sigma} - \tilde h_{\sigma} - f_{\sigma},
\]
which is a solution to the bi-Laplace problem
\[
\Delta^2 \tilde f_{\sigma}=0
\]
inside the layer $[0,L]^2\times [0,H]$ with boundary conditions
\begin{gather*}
\tilde f_{\sigma}=0,\, \partial_z\tilde f_{\sigma} -\sigma\partial_z^2 \tilde f_{\sigma}=0\quad \mbox{on }\{z=0\},\\
\tilde f_{\sigma}=-\tilde h_{\sigma},\, \partial_z\tilde f_{\sigma} +\sigma\partial_z^2 \tilde f_{\sigma} =0 \quad \mbox{on }\{z=H\}.
\end{gather*}
Controlling $\tilde f_{\sigma}$ globally as we did to bound $f_{\sigma}$ is not promising, because the boundary data at the top boundary are unbounded by Lemma \ref{L13}. Instead, in the following lemma, we solve this problem explicitly in Fourier space.

\begin{lemma}\label{L15}
The following estimate is true:
\[
\sup_{ [0,H]} \la (\partial_z \tilde f_{\sigma})^2\ra   \lesssim \sigma^2.
\]
\end{lemma}
\begin{proof}
By a tedious calculation, we   show that 
\[
\widehat{\tilde f}_{\sigma}(k,z) = - m_{\sigma}(k,z)\widehat{\tilde h}_{\sigma}(k,H),
\]
where the Fourier multiplier is given by the lengthy expression
\begin{align*}
 {m}_{\sigma}(k,z)
 &=
\left[(1+2|k|\sigma)^2\right.\\
&\qquad \left. -2e^{-2H|k|} \left(1+\left(2H^2 +8 H \sigma +4\sigma^2\right)|k|^2\right) +(1-2|k|\sigma)^2e^{-4H|k|}\right]^{-1}\\
 &\quad \times \left[e^{-|k|(H-z)} (1+2|k|\sigma)(1+2|k|\sigma +(1+|k|\sigma)|k|(H-z))  \right.\\
 &\qquad \left. +e^{-|k|(3H+z)} (1-2|k|\sigma)(1-2|k|\sigma -(1-|k|\sigma)|k|(H-z))  \right.\\
 &\qquad -e^{-|k|(3H-z)}\left((1-2|k|\sigma )^2 -H|k|(1-|k|\sigma)(1-2|k|\sigma)\right.\\
 &\qquad\quad \left. - \left(|k|-5|k|^2 \sigma+2|k|^3\sigma^2 -2H|k|^2(1-|k|\sigma)\right)z\right)\\
 &\qquad -e^{-|k|(H+z)}\left((1+2|k|\sigma)^2 +H|k|(1+|k|\sigma)(1+2|k|\sigma)\right.\\
 &\qquad\quad \left. \left. + \left(|k|+5|k|^2 \sigma+2|k|^3\sigma^2 +2H|k|^2(1+|k|\sigma)\right)z\right)\right].
\end{align*}
We notice that the denominator can be written as the sum of three  terms
\begin{align*}
\mel \left(1-2(1+2H^2|k|^2)e^{-2H|k|} +e^{-4H|k|} \right)+4\sigma|k| \left(1-4 H|k|e^{-2H|k|} -e^{-4H|k|}\right) \\
&+4\sigma^2 |k|^2 \left(e^{-2H|k|}-e^{-4H|k|}\right),
\end{align*}
which are nonnegative for   $H|k|\ge 2\pi H/L \gtrsim 1$ (again, in the limit as $\Ra\to \infty$). The denominator is thus bounded below   by $(1+\sigma |k|)^2$. For the derivative, we thus compute
\begin{align*}
\mel |\partial_z m_{\sigma}(k,z)|\\
 &\lesssim |k|e^{-|k|(H-z)} (1+|k|(H-z)) +  |k|e^{-|k|(3H+z)} (1+|k|(H-z))\\
&\quad+  |k|e^{-|k|(3H-z)} (1+|k|H + |k|^2H^2)+  |k|e^{-|k|(H+z)} (1+|k|H + |k|^2H^2)\\
&\lesssim |k|e^{-k(H-z)}(1+k(H-z)) +|k|e^{-k(3H+z)}(1+k(3H+z)) \\
&\quad +|k|e^{-2kH}(1+Hk)^2 +|k|e^{-kH}(1+kH)^2\\
&\lesssim |k|.
\end{align*}
 It thus follows that
\[
\sup_{[0,H]} \la (\partial_z\tilde f_{\sigma})^2\ra  \lesssim \left.\la |\grad_y\tilde h_{\sigma}|^2\ra\right|_{z=H} .
\]
The stated estimate now follows from Lemma \ref{L13}.
\end{proof}

We now easily control this corresponding term in the Nusselt number.

\begin{lemma}\label{L20}
The following estimate is true:
\[
\left|\frac1{\delta}\int_0^{\delta} \la \tilde f_{\sigma}\theta\ra\,\dd z\right|\lesssim \delta \sigma.
\]
\end{lemma}

\begin{proof}
The proof proceeds similarly as the one of Lemma \ref{L12}. We use the temperature bound \eqref{17}, the homogeneous Dirichlet boundary conditions at the bottom plate to apply the Poincar\'e inequality, and Jensen's inequality
\[
\left|\frac1{\delta}\int_0^{\delta} \la \tilde f_{\sigma}\theta\ra\,\dd z\right| \le \frac1{\delta}\int_0^{\delta}\la |\tilde f_{\sigma}|\ra\, \dd z\lesssim \int_0^{\delta}\la |\partial_z \tilde f_{\sigma}|\ra\, \dd z \le \delta \sup_{[0,1]}\, \la (\partial_z \tilde f_{\sigma})^2\ra^{1/2}.
\]
The statement follows now from Lemma \ref{L15}.
\end{proof}

We have now all information to bound the correction term $h_{\sigma}$.

\begin{prop}\label{P3}
Let $\sigma\le H$. Then the following estimate is true:
\[
\left| \frac1{\delta}\int_0^{\delta} \la h_{\sigma}\theta\ra\, \dd z\right| \lesssim \sigma \delta +\sigma^{\frac12}\delta^{\frac32} ,
\]
for any $\delta\in(0,1)$.
\end{prop}

\begin{proof}
We simply decompose
\[
h_{\sigma} = \tilde h_{\sigma} +f_{\sigma} +\tilde f_{\sigma},
\]
and infer the statement from Proposition \ref{P2} and Lemmas  \ref{L12} and \ref{L20} .
\end{proof}

\medskip

\noindent
\textbf{The correction term generated by $g_{\sigma}$.}
Instead of analysing $g_{\sigma}$ near the bottom boundary, it is enough to collect some information for $h_{\sigma}$ that we already derived on the top boundary. Indeed, if we stress the linear relation of both functions on the temperature $\theta$ by writing $h_{\sigma}=h_{\sigma}[\theta]$ (defined via \eqref{2}, \eqref{31}, \eqref{40}, and \eqref{8}) and $g_{\sigma}=g_{\sigma}[\theta]$ (defined via \eqref{32b}, \eqref{3}, \eqref{7}, \eqref{40}, \eqref{8}), it is not difficult to verify that one can be expressed by the other with the help of a symmetry relation: We have that
\[
g_{\sigma}[\theta] = h_{\sigma}^*[\theta^*],
\]
where we have set $\varphi^*(z) =\varphi(H-z)$. Via a change of variables, it then follows that
\begin{equation}\label{41}
\frac1{\delta}\int_0^{\delta}\la \theta g_{\sigma}[\theta]\ra\, \dd z = \frac1{\delta}\int_{H-\delta}^{H} \la \theta^* h_{\sigma}[\theta^*]\ra \,\dd z.
\end{equation}
Thanks to the maximum principle for the temperature \eqref{17}, it is thus sufficient to estimate $h_{\sigma}$ near the top boundary.

\begin{prop}\label{P4}
The following bound is true:
\[
\left|\frac1{\delta}\int_{H-\delta}^H\la \theta h_{\sigma}\ra\, \dd z \right|\lesssim \sigma \delta + \sigma^{\frac12}\delta^{\frac32},
\]
for any $\delta\in[0,1]$. 
\end{prop}

\begin{proof}
We use the temperature bound \eqref{17} and the Dirichlet boundary conditions in \eqref{31} to estimate
\[
\left|\frac1{\delta}\int_{H-\delta}^H\la \theta h_{\sigma}\ra\, \dd z \right|\le \int_{H-\delta}^{H}\la |\partial_z h_{\sigma}|\ra\, \dd z.
\]
We use our earlier decomposition $h_{\sigma} = \tilde h_{\sigma}+f_{\sigma}+\tilde f_{\sigma}$ and the triangle inequality to further bound the right-hand side
\[
\left|\frac1{\delta}\int_{H-\delta}^H\la \theta h_{\sigma}\ra\, \dd z \right|\le \int_{H-\delta}^{H}\la |\partial_z \tilde h_{\sigma}|\ra\, \dd z + \int_{H-\delta}^{H}\la |\partial_z f_{\sigma}|\ra\, \dd z + \int_{H-\delta}^{H}\la |\partial_z \tilde f_{\sigma}|\ra\, \dd z.
\]
The first term on the right-hand side is controlled using Lemma \ref{L13}. The second one is estimated by using Lemma \ref{L11} and repeating the argument of Lemma \ref{L12}. For the third term we invoke Lemma \ref{L15}. All the bounds that we obtain are as in the statement, which establishes this proposition.
\end{proof}

For completeness, we provide the argument for \eqref{15} in the following theorem.

\begin{theorem}\label{T3}
Suppose that $\sigma\le H$. Then for $H\gg1$, there is the estimate
\[
\Nu \lesssim 1 +\sigma^{1/2}
\]
\end{theorem}

\begin{proof}We recall the decomposition $w_{\sigma}-w_0= h_{\sigma}+g_{\sigma}$ from \eqref{42} and estimate the two individual terms in the Nusselt number bound \eqref{29} as in   Propositions \ref{P3} and \ref{P4}. We obtain that
\[
\Nu \lesssim \delta \sigma +\delta^{\frac32}\sigma^{\frac12} + \frac1{\delta},
\]
where $\delta\in(0,1]$ is arbitrary. First, if $\sigma \lesssim 1$, the optimal choice for $\delta $ is $\delta =1$, which then gives $\Nu \lesssim 1$. On the other hand, if $1\le \sigma \le H$, it is $\delta^{3/2}\sigma^{1/2} \lesssim \delta \sigma$ because $\delta\le 1$ by assumption, and then $\delta = \sigma^{-1/2}$ is optimal. This choice gives $\Nu \lesssim \sigma^{1/2}$. In either case, we have established the statement of the theorem.
\end{proof}

\subsection{The free-slip regime}

The strategy of our proof is the Whitehead--Doering bound, which we simplify here dramatically. 

\begin{theorem}[Free-slip dominated case]\label{T2}
For $H\gg1$ there is the estimate
\[
\Nu \lesssim 1+H^{1/4}.
\]
\end{theorem}

Before turning to the proof, we derive some global estimates on the vertical velocity component $w$.

\begin{lemma}\label{L6}
The following bounds are true:
\begin{align}
\int_0^H \la |\grad w|^2\ra\, \dd z &\le H \Nu, \label{32}\\
\int_0^H \la |\grad^2 w|^2\ra\, \dd z &\le H^{1/2} \Nu\label{33}.
\end{align}
\end{lemma}

\begin{proof}
We start with the proof of \eqref{32}. We test  the bi-Laplace equation \eqref{35} with $(-\Delta_y)^{-1} w$ and integrate  by parts twice. Using the Dirichlet boundary conditions in \eqref{22} and the periodicity in $y$, this leads to the identity
\begin{align*}
\mel 
\int_0^H \la |\grad_y w|^2\ra \, \dd z + 2\int_0^H \la (\partial_z w)^2\ra\,\dd z + \int_0^H \la |\grad_y^{-1}\partial_z^2 w|^2\ra\, \dd z + \left.\la \Delta_y^{-1} \partial_z w\partial_z^2 w\ra\right|_{z=0}^{z=H}\\
& = \int_0^H \la wT\ra\, \dd z.
\end{align*}
Using now the Navier-slip boundary conditions in \eqref{22}, we see that the boundary term has a sign,
\[
 \left. \la \Delta_y^{-1}\partial_z w\partial_z^2 w\ra\right|_{z=0}^{z=H}  =  \sigma \left.\la |\grad_y^{-1} \partial^2_z w|^2\ra\right|_{z=0,H}\ge0,
\]
and can thus be dropped. In view of the definition of the Nusselt number \eqref{36} and thanks to the boundary conditions for the temperature \eqref{21}, the inhomogeneity term above is identical to $H\Nu -1$. The estimate in \eqref{32} thus follows.

To derive \eqref{33}, we argue similarly. This time we test the bi-Laplace equation with $w$. Integrating by parts also in the inhomogeneity term and using the Cauchy--Schwarz inequality, we eventually arrive at
\begin{align*}
\int_0^H \la |\grad^2 w|^2\ra\, \dd z +  \sigma \left. \la (\partial_z ^2w)^2\ra\right|_{z=0,H} & \le \left(\int_0^H \la |\grad_y w|^2\ra\, \dd z\int_0^H\la |\grad_y T|^2\ra\, \dd z\right)^{1/2}.
\end{align*}
We drop the boundary term again, use the previous estimate and the representation \eqref{23} of the Nusselt number.
\end{proof}

We may now proceed with the proof of Theorem \ref{T2}.

\begin{proof}[Proof of Theorem \ref{T2}]
We start again with the local Nusselt bound \eqref{1}, use the maximum principle for the temperature $|T|\le 1$ and the Dirichlet boundary conditions for the vertical velocity component,
\[
\Nu \le \frac1{\delta}\int_0^{\delta} \la |w|\ra\, \dd z +\frac1{\delta} \le \int_0^{\delta} \la |\partial_z w|\ra\, \dd z + \frac1{\delta} \le \delta \sup_{[0,H]} \la (\partial_z w)^2\ra^{1/2} + \frac1{\delta}.
\]
The crucial observation by Whitehead and Doering was that, here, we may use the Dirichlet boundary conditions to realise that because of the identity
\[
\int_0^H \partial_z w(r,y,z)\, \dd z = 0,
\]
there must exists a $\tilde z=\tilde z(t,y)$ such that $\partial_z w(t,y,\tilde z)=0$. Using this information, we may invoke the fundamental theorem to observe that
\[
(\partial_z w(z))^2  = 2 \int_{\tilde z}^z \partial_z w\partial_z^2w\, \dd z \le 2\left(\int_0^H (\partial_z w)^2\, \dd z\int_0^H (\partial_z^2 w)^2\, \dd z\right)^{1/2},
\]
and thus, averaging in time and horizontal space and exploiting Lemma \ref{L6}, we have the bound
\[
\sup_{[0,H]} \la (\partial_z w)^2\ra \lesssim \left(\int_0^H\la |\grad w|^2\ra\, \dd z\int_0^H \la |\grad^2 w|^2\ra\,\dd z\right)^{1/2} \le H^{3/4}\Nu.
\]
Substituting this estimate in the previous Nusselt bound gives
\[
\Nu \lesssim \delta H^{3/8}\Nu^{1/2} +\frac1{\delta}.
\]
Optimizing in $\delta $ yields $\delta \sim H^{-1/4}$, which leads to  the result.
\end{proof}

\section*{Acknowledgement}
The author thanks Fabian Bleitner, Lukas Niebel and Camilla Nobili for stimulating discussions.	This work is funded by the Deutsche Forschungsgemeinschaft (DFG, German Research Foundation) under Germany's Excellence Strategy EXC 2044--390685587, Mathematics M\"unster: Dynamics Geometry Structure.

\bibliography{rbc_lit}
\bibliographystyle{abbrv}

\end{document}